\documentclass[a4paper,12pt]{article}
\usepackage{amsmath,amssymb, amsfonts, amsthm}
\usepackage{enumerate}
\newtheorem{theorem}{Theorem}
\newtheorem{corollary}[theorem]{Corollary}

\newtheorem{proposition}[theorem]{Proposition}

\newtheorem{remark}[theorem]{ Remark}

\renewcommand{\Im}{\mathrm{Im}}
\renewcommand{\Re}{\mathrm{Re}}
\newcommand{\Int}{\displaystyle \int}

\def\HH{{\mathcal H}}

\def\R{{\mathbb R}}

\newcommand\dint{\displaystyle\int}

\newcommand\dlim{\displaystyle\lim}

\begin{document}
\title {\bf Stark resonances in   a  quantum   waveguide with analytic curvature}
\author{Philippe Briet \footnote{e-mail: briet@univ-tln.fr}, \\
{\small   Aix-Marseille Universit\'{e}, CNRS, CPT UMR 7332, 13288 Marseille, France, and } \\ 
{ \small Universit\'{e} de Toulon, CNRS, CPT UMR 7332, 83957 La Garde, France}\\
and \\
Mounira Gharsalli  \footnote{e-mail: gharsallimounira@gmail.com}\\
 {\small $^2$Laboratoire EDP, LR03ES04, D\' epartement de Math\' ematiques, Facult\'e des Sciences de Tunis,} \\
{  \small Universit\' e de  Tunis El Manar, El Manar 2092 Tunis, Tunisie}}
\noindent

\date{}
\maketitle

\medskip

MSC-2010 number: 35B34,35P25, 81Q10, 82D77
\begin{abstract} 
We investigate the influence of an electric field on   trapped modes arising in   a two-dimensional curved quantum waveguide ${\bf \Omega}$ i.e. bound states of  the corresponding   Laplace operator  $-\Delta_{ {\bf \Omega} }$.  Here the curvature  of the guide is supposed to satisfy  some assumptions  of analyticity,  and  decays as $O(|s|^{-\varepsilon}), \varepsilon >3$ at infinity.  We show that  under conditions on the electric field  $ \bf F$, ${\bf H}(F):= -\Delta_{ {\bf \Omega} } + {\bf F}.  {\bf x} $ has  resonances near  the discrete eigenvalues of $-\Delta_{ {\bf \Omega} }$.
\end{abstract}
\section{Introduction}

This paper is a continuation and extension of earlier work 
\cite{BG}. 
 Let us recall  the problem; for more details we refer to \cite{BG, Herbst}. The object of our interest is the Stark operator 
 \begin{equation}
 \label{HF}
{ \bf H}({F})=-\Delta_{ {\bf \Omega} }+ {\bf F}\cdot {\bf x },\,\, {\bf F} \in  \mathbb R^2, \; {\bf x }=(x,y) \in \bf \Omega
\end{equation}
 on $L^{2}(\bf\Omega)$ where  $\bf\Omega$ is  a curved strip in $\R^{2}$  of constant width $d>0$  defined around a smooth  curve $\Gamma$.  
 The operator $-\Delta_{ {\bf \Omega} }$ is defined in a standard way   by means of DBC on  the boundary of ${\bf \Omega}$, $\partial  \bf\Omega$ \cite{reed simonIV}.
 
  We assume that $\bf \Omega$ is not straight,
let $ s \in \mathbb R \to \gamma(s) $ be  the signed curvature of  $\Gamma$.  In \cite{BG} it is  supposed that     $ \gamma \in  C_0^2(\mathbb R)$ and  $d\Vert \gamma\Vert_\infty <1 $. Here we consider a more general situation  namely we assume  that
 \begin{itemize}
\item[(h1)] $ \gamma \in  C^2(\mathbb R)$ and   there exist $ a_0, r_0 >0$  s.t. $\gamma$  has   an analytic extension in 
$$ {\mathcal O}_{a_0, r_0} =\{z \in \mathbb C, \; \vert  \arg z \vert < a_0 \}  \cup \{z \in \mathbb C, \; \vert \pi -  \arg z  \vert < a_0 \} \cap  \{  \vert \Re z  \vert > r_0 \}.$$    Moreover  $\gamma$  satisfies  $d\| \Re\gamma\|_{\infty}< 1$.

\item[(h2)]  There exists $\varepsilon > 3 $ s.t. $\gamma(z)= O(|z|^{-\varepsilon})$  as $ \vert  \Re z\vert \to
\infty.$
\end{itemize}
 Then  $\bf \Omega$ is asymptotically straight. We choose the lower boundary of $\bf \Omega$  near $s=-\infty$  as the reference curve. Introduce  orthogonal coordinates $(s, u) \in \Omega:= \mathbb R \times  (0,d)$,   related to   $(x,y) \in \bf \Omega$   via the relations  \cite{ES}, 
  \begin{equation} \label {CV}
 x(s,u)=  \int_0^s \cos (\alpha(t))dt - u \sin(\alpha(s)), \; 
 y(s,u)=  \int_0^s \sin (\alpha(t))dt + u \cos(\alpha(s))
 \end{equation}
 where $\alpha(s)=\dint_{-\infty}^{s}\gamma(t)\,dt$. Set  $\alpha_{0}=\dint_{-\infty}^{+\infty}\gamma(t)\,dt$.
 
Since  we have supposed $ d  \Vert \gamma \Vert_\infty <1$, the operator $ -\Delta_{ {\bf \Omega} }$  is  unitarily equivalent to
 \begin{equation}
\label{hamilt}
H=H_{0} + V_0,  H_0 =T_s +T_u
\end{equation} 
in the  space $L^{2}(\Omega)$,  with DBC  on $ \partial \Omega$ \cite{ES}, where
\begin{equation}  \label{g}
T_{s}:= -\partial_{s}g\partial_{s},\; \; T_u: = -\partial_{u}^{2}, \; \;  g(s,u) =(1+u \gamma(s))^{-2},
\end{equation}
 and
 \begin{equation} \label{V0}
V_{0}(s,u)=-\frac{\gamma(s)^{2}}{4(1+u\gamma(s))^{2}}+\frac{u\gamma''(s)}{2(1+u\gamma(s))^{3}}-\frac{5}{4}\frac{u^{2}\gamma'(s)^{2}}{(1+u\gamma(s))^{4}}.
\end{equation}
 With our assumptions, the potential $V_0$  is bounded and then   $ H= H_0 + V_0$ is a  self-adjoint  operator   with domain
  \cite{G,Kri},
  \begin{equation} \label {DH}
 D(H)  = D(H_0)=  {\HH}_0^1 ( \Omega) \cap {\mathcal H}^2(\Omega).
 \end{equation}
Here and hereinafter we use standard notation for  Sobolev space.  Moreover the essential spectrum of this operator, $ \sigma_{ess}(H)= [\lambda_0, + \infty),\,\,\, $ where   $\{ \lambda_0, \lambda_1,.... \}$ are  the transverse modes of the system i.e.  the eigenvalues of  the operator $-\partial_u^2 $ on $ L^2(0,d)$ with DBC on the boundary $\{0,d\}$ \cite{ED}.  \\
Denote the exterior field as  ${\bf F}=F(\cos(\eta), \sin(\eta))$. With respect to the new coordinates,  the  field interaction  is then 
  \begin{equation} \label{W1}
   W(F)(s,u):=  {\bf F}\cdot {\bf x } = F\dint_{0}^{s}\cos\left(\eta-\alpha(t)\right)\,dt + F u \sin\left(\eta-\alpha(s)\right).
 \end{equation}
 
Here  we study a field regime which was  not considered so far   i.e. the intensity of the field  is the free parameter  in $ 0<F< 1$ and the direction $\eta$ is  fixed satisfying
\begin{equation}
\label{eta}
 |\eta|<\frac{\pi}{2}\,\,\,\,\,\,\mbox {and}\,\,\,\,\,\, |\eta-\alpha_{0}|>\frac{\pi}{2}.
 \end{equation}
As  discussed in the Section 2,  this  implies that   $ W(F)(s,u) \to   -\infty  $ as $ s \to \pm \infty $.   Thus 
the non trapping region for a given negative energy $E$  contains both  a neighbourhood of $s=-\infty$ and $s = \infty$.

 We denote   by $H_0(F)  =T_s +T_u + W(F)$ and $H(F)  =H_0(F)  +V_0$. Then a straightforward extension of the Theorem $2.1$  of \cite{BG} shows that  for $F>0$, the Stark operator  $ H(F)$ is essentially self-adjoint
  and $\sigma(H(F))=\R$.\\
We are interested in  the study  of the influence of the electric field on    the discrete spectrum  of $H$.   We   want to show  that the eigenvalues  of $H$ give rise to resonances for the Stark operator $ H(F), F>0$.
The resonances of $H(F)$  are understood in the standard way \cite{AC, Hu, reed simonIV}. Evidently if $ H $ has no  discrete eigenvalue below $\lambda_0$ then  this  result 
 proves  that $ H(F)$  has neither resonance or embedded eigenvalue  in $\{ z \in \mathbb C, \Re z <  \lambda_0 \}$.
 For  a discussion  about  eigenvalues of $H$ we refer the reader to \cite{ED}. \\

To study this problem we need an additional assumption,
\begin{itemize}
\item[(h3)] $ \Im \gamma (z ) \geq 0 $ for $ z \in  { \mathcal O}_{a_0, r_0}, 0\leq \arg z \leq a_0$ or $0 \leq \arg z -\pi \leq a_0$.
 \end{itemize}

\begin{remark} 
i)  In fact it is only  necessary to suppose that  the product $ u\Im \gamma (z ) \geq 0 $ for $ z \in  {\mathcal O}_{a_0, r_0}, 0\leq \arg z \leq a_0$ or $0 \leq \arg z -\pi \leq a_0 $. So the case $Im \gamma (z ) \leq 0 $ is reducing to the present one 
 by taking the  other boundary as reference curve. \\
ii) We can  check that   $ \gamma (s) = \frac{\alpha }{1+s^{2n}}; \; n \geq 2 $,  $ \alpha < 0$ satisfies our assumptions with $r_0 >1$ and $ 0 < a_0\leq \frac{\pi}{4n}$.

\end{remark}



The study  of   the Stark effect  was  considered by several authors, see e.g. \cite{PB, Hislop}  for a discussion  concerning  the case of  Shr\" odinger operators on $\mathbb R^n$ and  \cite { BG, Ex}  for  operators defined   on curved strips. In particular in   \cite {BG},   under assumptions   on the    curvature  above mentioned and if  $  |\eta|<\frac{\pi}{2}\,\,\mbox {and}\,\,\, |\eta-\alpha_{0}|<\frac{\pi}{2}$,  it is  proved  the existence of   Stark  resonances  having an  width  exponentially  small   w.r.t. $F$ as $F$  tends to zero. 
In this paper we would like  to extend this result under weaker assumptions on  $\gamma$ i.e. hypotheses  (h1-3)  and in the field regime \eqref{eta}. \\ More precisely we will show the following. 

 \begin{theorem}  \label{t0}  Suppose (h1-3).  Let  $E_{0}$ be  an  discrete eigenvalue   of $H $ of  finite multiplicity $ n \in \mathbb N$.  There exist  $F_0 >0 $ and a  dense  subset  $ \mathcal A $  of $L^2(\Omega) $    such that  
 for   $ 0<F \leq  F_0 $ \\
$$ i)  \quad
\;  z \in \mathbb C, \Im z >0 \to  {\cal R}_{\varphi}(z)=  \big( ( H(F)-z)^{-1}\varphi, \varphi),\;  \varphi \in \mathcal A  $$
has an meromorphic extension in  a complex neighbourhood  $ \nu_{E_0}$ of $E_0 $,  through  the cut due to  the  presence  of  the continuous spectrum of $H(F)$. \\
ii)  $ \cup_{\varphi \in  \mathcal A}\{   { \rm poles \; of }  \; {\cal R}_\varphi(z) \} \cap \nu_{E_0}$ contains  $n$ poles $Z_0(F),...Z_{n-1}(F)$  converging to $E_0$ when $F\to0$.\\
iii)  Suppose that $E_0$ is a simple eigenvalue of $H$, Let $Z_0(F)$  as in ii). For Then there exist  two constants   $ 0< c_1,c_2$ such that for $0< F \leq  F_0 $,
$$  \vert \Im Z_0 \vert \leq  c_1 e^{-\frac{c_2}{F}}. $$
\end{theorem}
In this paper  we only give elements  we need to extend the strategy  of \cite{BG} to the situation we   now  consider. In particular Theorem 1.2 iii) is covered by  \cite[Section 6] {BG}, so we omit the proof here.

The plan of this work is as follows. In section 2  we introduce a  local  modification of the operator  $H(F)$ we need to  perform the meromorphic continuation of the resolvent of $H(F)$.    Some elements of the the complex distortion theory are given in  the Section 3.   In section 4 we  define  the  extension of the resolvent of $H(F)$, this  allows  to    define   the resonances of $H(F)$. The existence of resonances is proved  in the Section 6. The section 7 is devoted to some concluding remarks.
\section{The reference operator} \label{TRO}
 To prove the theorem \ref{t0} we use the distortion theory such  that it can be found in     \cite {BG, PB}. The first step is  to  consider  a local modification of the operator $H_0(F)$  called  the   reference operator. It is  defined  as follow.  Note that 
\begin{equation}   \label{W201}{\rm if }\; s <0, \; W(F,s,u) = F\big ( s \cos(\eta) + u \sin ( \eta) +   A_- \big)  + R_-(F,s), \;   
 \end{equation}
 where
 $  R_-(F,s)= F \big( \Int_{-\infty }^s (\cos(\eta-\alpha(t))-\cos(\eta))dt + u(\sin(\eta-\alpha(s))-\sin(\eta))\big).$  
\begin{equation}  \label{W202}
 {\rm{if }}\; s \geq 0, \; W(F,s,u) = F \big( s  \cos(\eta-\alpha_0)  + u \sin (\eta- \alpha_0) +   A_+ \big) + R_+(F,s) 
 \end{equation}
 $$   {\rm{where  }}\; R_+(F,s)=F \big( \int_{s}^{ \infty}(\cos(\eta-\alpha_0)-\cos(\eta-\alpha(t))dt + u(\sin(\eta-\alpha(s))-\sin(\eta-\alpha_0))\big).$$

The constants  $A_-, A_+$ are
 $$ A_-:= \Int_{-\infty}^0 ( \cos(\eta)- \cos(\eta-\alpha(t))dt, A_+:=   \Int_{0}^{ \infty}(\cos( \eta -\alpha(t))-  \cos(\eta - \alpha_0)dt. $$
 
 In view of  (h2) and 
   \begin{equation} \label{alphainfty}
  \alpha(s)  = O(\frac {1}{\vert s \vert^{\varepsilon-1}}), \; { \rm{as}}\; s \to -\infty;\;  \alpha(s)  =  \alpha_0 +O(\frac {1}{\vert s \vert^{\varepsilon-1}}), \;  \rm{as}\; s \to  \infty,
  \end{equation} 
   $A_-,A_+$ are well defined. Moreover we have, 
  \begin{align}  \label{W2}
   R_-(F,s), R_+(F,s) =  O( \frac{F}{ \vert s \vert ^{\varepsilon-2} }), \;{\rm{  \; as}}\; s \to \pm \infty.
 \end{align}
 Set  $  R_-(F,s)= 0$ for $s\geq  0$ and   $  R_+(F,s)= 0$ for $s< 0$.\\
  Hence   it is  quite natural to consider     a modified interaction   defined as,
 $$\tilde W(F)(s,u)  =   F( s \cos(\eta) + u \sin( \eta ) +   A_-)\;   {\rm{for}}\; s <0 $$
 and 
 $$\tilde W(F)(s,u)  =   F( s \cos(\eta-\alpha_0)  + u \sin( \eta - \alpha_0) + A_+)\; \;   {\rm{for}}\;s  \geq 0.$$

In particular we get
 \begin{equation} \label{dW1}
 W(F)-   \tilde W(F) =  R(F):= R_+(F) + R_-(F)=  O( \frac{F}{ \vert s \vert ^{\varepsilon-2} }) \; { \rm{as}}\; s \to \pm \infty.
 \end{equation}
Notice also that 
 \begin{equation} \label{dW2}
  R'(F)=  O( \frac{F}{ \vert s \vert ^{\varepsilon-1} }),  \; R''(F)=  O( \frac{F}{ \vert s \vert ^{\varepsilon} })  { \rm{as}}\;  \;s \to \pm \infty .
 \end{equation}
Let  $ \tilde H_0(F)$ be the reference  operator in $\L^2(\Omega)$,
$$ \tilde H_0(F) := H_0 + \tilde W(F). $$ 
$ \tilde H_0(F)$  differs from $ H_0(F)$  by an additional  bounded operator, then it is essentially self-adjoint and $ \sigma(\tilde H_0(F))= \mathbb R$.

\section{Complex distortion}

In this  section,  we give   necessary  elements  for  the complex distortion theory  \cite{ AC,  PB,Hu} we need  to define    the family of distorted operators $ \{H_{\theta}(F), F \leq F_0\} $, for some $F_0 >0$  and complex values of $\theta$.  We denote by $\beta= \Im \theta$. \\

Let $E\in\R$, $E<0$ be the reference energy   and $ 0<\delta E< \frac {1 }{2} \min  \{ 1,\vert E \vert \}$. Denote 
$E_{-}= E - \delta E $ and  $E_{+}=E + \delta E$. We choose  a real function $\phi\in\mathrm{C}^{\infty}(\mathbb{R})$ such that:
 \begin{equation}\label{exem}
 \phi(t)=\left\{%
\begin{array}{ll}
1 & \hbox{if $t<E$, } \\
0&\hbox{if $t> E_+$} 
\end{array}%
\right.
\end{equation}
and satisfying $\Vert \phi^{(k)} \Vert_\infty = O((\frac {1}{\delta E})^{k})$.\\
For $\theta \in \mathbb R$, we introduce the distortion on $\Omega$, $s_{\theta}(s,u):=(s + \theta f(s),u)$
 where 
 \begin{equation}\label{field}
 f(s)=\left\{%
\begin{array}{ll}
 -\frac{1}{F cos(\eta)} \Phi_-(s) & \hbox{if $s \leq 0$, } \\
 -\frac{1}{F\cos( \eta- \alpha_0)} \Phi_+(s)&\hbox{if $s > 0$} 
\end{array}%
\right.
\end{equation}
where  $\Phi_ -(s)= \phi \big(F \cos (\eta) s \big)$ and   $\Phi_+(s)= \phi \big(F \cos (\eta- \alpha_0)s  \big)$. Set  $\Phi(s)= \Phi_ -(s) + \Phi_ +(s)$.

So for $k \geq 1$,
$\Vert \Phi^{(k)} \Vert_\infty \leq \left(\frac{F}{\delta E}\right)^{k},\quad \Vert f^{(k)} \Vert_\infty \leq \frac{F^{k-1}}{(\delta E )^{k}}.$

In view of  the definition \eqref{field}, for small $F$,  $s_\theta $ is an translation
along the longitudinal axis in  neighbourood of $ s=\pm \infty$ since
\begin{equation} \label{f2}
f(s) =- \frac{1}{F\cos(\eta)} \; {\rm{for }}\; s \leq \frac{E}{ F \cos \eta} \;  {\rm{and  }}\; 
 f(s) =- \frac{1}{F\cos(\eta-\alpha_0)} \;{  \rm{for }}\; s \geq \frac{E}{\cos (\eta- \alpha_0)}.
\end{equation}

Assume $|\theta|< \delta E$. Let $U_{\theta}$ be the operator defined on $L^2(\Omega)$ by
\begin{equation} \label{dst}
U_{\theta}\psi(s,u) = (1+\theta f')^{\frac{1}{2}}\psi\left( s_{\theta}(s,u)\right).
\end{equation}
The operators $U_{\theta}$ are unitary and generate the family,
\begin{equation}
\label{unitaire}
 H_{ \theta}(F)=U_{\theta}H(F) U_{\theta}^{-1}= H_{0, \theta}(F) + V_{0,\theta}
\end{equation}
where 
\begin{equation} \label{HOtheta}
H_{0, \theta}(F):=T_{s,\theta} + T_{u} +  W_{\theta}(F),
\end{equation}
with
\begin{equation}
\label{Tseta}
 T_{s,\theta}=-(1+\theta f')^{-\frac{1}{2}}\partial_{s}(1+\theta f')^{-1}g_{\theta} \partial_{s}(1+\theta f')^{-\frac{1}{2}},
\end{equation}
 $g_\theta=  (1 +u \gamma_\theta)^{-2}, \gamma_\theta= \gamma \circ s_{\theta}$, \,$  W_{\theta}(F)=   W(F)\circ s_{\theta}$ and $V_{0,\theta}:=V_{0}\circ s_{\theta}. $ \\
We also have
\begin{equation}
\label{ts teta}
 T_{s,\theta}=  -\partial_{s}(1+\theta f')^{-2}g_\theta \partial_{s} + S_{\theta},
\end{equation}
where
\begin{equation}
\label{Seta}
S_{\theta}=-\frac{5g_\theta}{4}\frac{\theta^2f''^{2}}{(1+\theta f')^4}+ \frac{g_\theta}{2}\frac{\theta f^{'''} }{(1+\theta f')^{3}}  +\frac{g'_{\theta}}{2}\frac{\theta f^{''} }{(1+\theta f')^{3}}.
\end{equation}
For  small $F$,the analytic  extension of  the family $\{ H_{ \theta}(F),  \theta \in \mathbb R,  \vert \theta \vert < \delta E\}$  to  a  complex disk  $ \vert  \theta \vert \leq \beta_0$ for some $\beta_0 >0$ depends strongly  on the 
analytic property  of the dilated curvature $\gamma_\theta$. But clearly      (h1) and the definition  \eqref{dst} of $s_\theta$ imply that if  $0<F\leq F_0$  with $F_0$ is  small enough,  $s \in \mathbb R$,  $ \theta \to \gamma_\theta(s)$ is   analytic in  $ \vert \theta \vert \leq  a_0 $.\\

Then the  same arguments  as  in the proof of  \cite  [Proposition (3.1)] {BG}   lead to
\begin{proposition}
\label{l1}Suppose (h1-2). 
There exist $F_0>0$ and  $0<\beta_0 \leq \min\{ \delta E,  a_0 \}$ such that  for all $0< F\leq F_0$, 
$\{  H_\theta(F);   \vert \theta  \vert \leq  \beta_0\}$ is a self-adjoint  analytic family of operators.
 \end{proposition}
It should be noted that  all the critical values $ \beta_0$ and $ F_0$  that appear  in this article are independent from each other.\\

In a similar way let  $\tilde H_{0, \theta} (F) = H_{0, \theta} +    \tilde W_\theta (F)$, where $ \tilde  W_{\theta}(F) =   \tilde W_{\theta}(F)\circ s_\theta$. Since
 \begin{equation}
\label{dWtheta}
R_\theta(F) =W_\theta(F)- \tilde W_\theta(F)= R(F) +  i\beta f R'(F)+ O( \frac {F  \beta^2 f^2}{ \vert s\vert ^{  \varepsilon}}),
\end{equation}
then  from \eqref{dW1}  and \eqref{dW2}  and (h1),  the multipliers $ W_\theta(F)-\tilde W_{\theta}(F)$  and    
$V_{0,\theta}$  have an analytic extension as bounded operators in $ \vert \beta \vert  \leq  \beta_0$. Hence we get
 
 \begin{corollary}
\label{l1}
 For all $0< F<F_0$ 
the family of operators $\{ \tilde H_{0,\theta}(F) ;   \vert  \theta \vert   \leq  \beta_0\}$ is  a self-adjoint  analytic family of operators.
 \end{corollary}

For  some technical points of the Section 6 below, we need to introduce   an another family of  operators on $ L^2(\Omega)$,
let $ \theta \in \mathbb R , \vert  \theta \vert \leq \delta E$ and 
 \begin{equation}
\label{HOA}
H_{0,\theta} = U_\theta ( T_s + T_u) U^{-1}_\theta = T_{s, \theta}+ T_u,
\end{equation}
Note that $H_{0,\theta} \not= H_{0,\theta}(F=0)$.  Indeed the distortion is supported in $\{ \vert s \vert \geq c'/F\}$ for some $c'>0$,  thus   {\it at least formally} $ H_{0,\theta}(F=0) = H_0$. \\
Under our assumptions  (h1) and (h2).  Following    arguments  of  \cite  [Proposition (3.1)] {BG} evoqued above,  there exist $F_0>0$ and  $0<\beta_0 \leq \min\{ \delta E,  a_0 \}$ such that  for all $0< F<F_0$, 
$\{  H_{0,\theta};   \vert \theta  \vert \leq  \beta_0\} , D(H_{0,\theta}) = D(H_0) $ is  also a self-adjoint analytic family of operators.\\
In fact the main property of  $H_{0,\theta}$  we use in  the Section 6 below   is 

\begin{proposition}
\label{HOAp}Suppose (h1-3) hold.  Then  there exists $ 0< \beta_0$ and $0<F_0$ such that  for  $0 <  F \leq  F_0$,   $\{  H_{0,\theta}; \beta= \Im \theta\geq 0,  \vert \theta \vert  \leq  \beta_0\leq \min\{ \delta E,  a_0 \}\}$
is a family of sectorial operators with a sector  contained in 
$$ {\mathcal S}= \{ z \in \mathbb C,  -2 c \beta \leq  \arg( z -\lambda_0 + \zeta ) \leq 0 \},$$
for some strictly positive  constant $c$. Here $ \zeta $ is an error term with $\Re \zeta, \Im \zeta \geq 0$ and  $ \vert \zeta \vert = O(\beta F^2)$.
 \end{proposition}
\proof   We may suppose $\theta =i \beta, 0 \leq   \beta  \leq  \beta_0 $. Note first that for $ \beta $ small enough
then
 \begin{equation}
\label{gamtheta}
\gamma_\theta(s)= \gamma ( s + i \beta f(s)) = \gamma(s)+  i\beta f \gamma'(s)+ O( \frac { \beta^2 f^2}{ \vert s\vert ^{ \varepsilon}}),
\end{equation}
By using \eqref{ts teta} we have  for $ \varphi \in  D(H_0), \; \Vert  \varphi \Vert =1$, 
 \begin{equation} \label{11}
  ((H_{0, \theta} -\lambda_0)\varphi, \varphi ) = (G_\theta \partial_{s} \varphi,  \partial_{s}\varphi)   +  (S_{\theta}\varphi, \varphi) + ((T_u-\lambda_0) \varphi, \varphi)  
  \end{equation}
where $G_\theta = (1+\theta f')^{-2}g_\theta$. Consider $ q:=  (G_\theta \partial_{s} \varphi,  \partial_{s}\varphi)- ((T_u-\lambda_0) \varphi, \varphi) $.  By using (h2), we have for $F$ and $\beta$ small,
\begin{align}\notag  \Re G_\theta^{-1}&=  ((1+u \Re \gamma_\theta)^2 -(u\Im \gamma_\theta)^2)(1- \beta^2f'^2) -4\beta f'u\Im \gamma_\theta (1+ \Re \gamma_\theta)  \\    \label{RG}  & \geq  (1+u \Re \gamma_\theta)^2  (1- O(\beta^2))  +O(\beta^2F^{\varepsilon_1}) \geq c_1,
 \end{align} 
  for some  constant  $c_1>0$ and then   $\Re q \geq c_1 \Vert  \partial_{s} \varphi \Vert $.  In the other hand,  
 \begin{align}  \label{IG} \Im G_\theta^{-1}= 2  \big( \beta f' \big( (1+u \Re \gamma_\theta)^2 -(u\Im \gamma_\theta)^2 \big) 
 &+ u\Im \gamma_\theta (1+ \Re \gamma_\theta)(1- \beta^2f'^2)\big). 
 \end{align} 
Since $f'\geq 0$  and  by (h3),  $\Im \gamma_\theta \geq 0$ then    for $\beta$ and $F$ small $Im G_\theta^{-1}\geq 0$,
 hence $ \Im q \leq 0$. In the other hand in view of \eqref{gamtheta},  it is straightforward to check that for $F$ and $\beta$ small enough there exist $ c_2 >0$ such that  $ \vert  \Im q  \vert  \leq 2\beta  c_2 \Vert  \partial_{s} \varphi \Vert $. Then   by \eqref{RG},  
   \begin{equation} \label{q111111}
   q \in \ \{ z \in \mathbb C,  -2 c \beta \leq  \arg( z -\lambda_0 ) \leq 0 \}
     \end{equation}
  for some stricly positive constant $c$. But we know that
  $(S_{\theta}\varphi, \varphi) = O( \beta F^2)$,  then \eqref{11}   together with   \eqref{q111111} conclude  the proof of  the Proposition \ref{gamtheta}. \qed
  
\begin{remark}

The main point in the proof of the Propositions \ref {HOAp}  and  \ref{th} below,   is the  fact that    $ \Im \gamma_\theta \geq 0$  which is insured by (h3). In view of \eqref{gamtheta},  then a necessary condition  to satisfy  this condition   for $\beta $ small  is given by 
$$ f(s) \gamma' \geq 0.$$
This means that  the curvature has to satisfy $  \gamma' \leq 0 $ in a neighbourhood of $ s = -\infty$
and $ \gamma' \geq 0 $ in a neighbourhood of $ s = \infty$.  \\
Roughly speaking this last inequality is a  geometrical non trapping estimate,  in the spirit of those given in \cite{PB}   for  the case of   electric perturbations.
\end{remark}

\section{Spectral estimates}

The main result in this section is the following.   Let  $ \theta =i\beta $.
$f^{\sharp}= \Phi-1$ where $\Phi$ defined above in the Section $2$, clearly $f^{\sharp}<f'$. Set $\mu_\theta=1+\theta f^{\sharp}$
and
\begin{equation}
\label{nutheta}
\nu_{\theta}=\left\{z\in \mathbb{C},\,\Im\mu_{\theta}^{2}(E_{-} + \lambda_{0}-z)< \beta\frac{\delta E}{4}\right\},
\end{equation}
In the sequel we denote by  ${\nu_\theta}^c$ the  complement set of 
${\nu_\theta}$.  Let  $ \rho(H_{\theta}(F))$ be  the resolvent set of  $H_{\theta}(F)$.  Consider first
$\theta = i \beta$ then

\begin{proposition} Let $E<0$ be an  reference energy. 
\label{th}
Suppose (h1-3). Then there exists $0<\beta_0 \leq  \min\{ \delta E,  a_0 \}$ and $F_0>0$  such that for $ 0 < \beta \leq \beta_0$, 
   and   $0<F\leq F_0$,   
\begin{itemize}
\item [(i)]  $ \nu_\theta \subset \rho( \tilde H_{0,\theta}(F))$
\item[(ii)] For $z\in \nu_\theta$, $\|( \tilde H_{0,\theta}(F)-z)^{-1}\| \leq {\rm dist}^{-1} (z, \nu^c_{\theta})$.
\end{itemize}
\end{proposition}
\begin{proof}
 Note that  in view of  (\ref{Tseta}), we have
$$ \mu_\theta T_{s,\theta} \mu_\theta= T_1(\theta)+ i T_2(\theta) + \mu_\theta\left(T_{s,\theta} \mu_\theta\right), $$
where $T_{1}(\theta)=-\partial_{s}\Re\{ \mu_{\theta}^{2} (1+ \theta f')^{-2} g_\theta\} \partial_{s} $ and 
$T_{2}(\theta)=-\partial_{s}\Im\{ \mu_{\theta}^{2}(1+ \theta f')^{-2} g_\theta \}\partial_{s}$.
 $T_{1}(\theta), T_{2}(\theta)$ are  symmetric  operators,  let us check that  under our assumptions $T_2(\theta)$ is actually  negative.\\
It then sufficient to show that $ q':= \Im \mu_\theta^{2}(1-i\beta f')^{2}(1+u \bar \gamma_\theta)^2 \leq 0$. We have 
\begin{align*}
q'
= &2\beta\left(f^{\sharp}-f'\right)\left(1- \beta^2f'f^{\sharp}\right)\left((1+u\Re{\gamma_\theta})^2 - u^2 \Im {\gamma_\theta}^2 )\right )\notag\\
& -  2u\Im{\gamma_\theta}(1+u\Re{\gamma_\theta})\left((1-\beta^2(f^{\sharp})^{2})(1-\beta^2(f')^{2})+ 4 \beta^2f^{\sharp}f'\right)\notag
\end{align*}
By  using   \eqref{gamtheta} together with (h1), for $F$ and $\beta$ sufficiently small  we have 
$$ \left(1- \beta^2f'f^{\sharp}\right)\left((1+u(\Re{\gamma_\theta})^2 - u^2 \Im {\gamma_\theta}^2 \right )) \geq 0$$ 

and
$$(1-\beta^2(f')^{2})(1-\beta^2(f^{\sharp})^{2})+ 4 \beta^2f^{\sharp}f'\geq 0.$$
We know that  $f^{\sharp}<f'$, therefore,   in view of  (h3) we get  our claim.\\
 In the other hand    we have 
 \begin{align}
\Im \mu_\theta^2\left(  \tilde W_\theta(F)-E_{-}\right)&= \left(1-\beta^2 {f^{\sharp}}^2\right)\Im  \tilde W_\theta(F)+ 2 \beta f^{\sharp}\left( \Re  \tilde W_\theta(F)-E_{-}\right)\notag\\
&= -\beta \left ( (1-\beta^2(f^{\sharp})^2) \Phi 
+ 2  (1-\Phi)\left(\tilde W(F)-E_{-}\right)\right).\notag
\end{align}
  To estimate the r.h.s. of this expression, we note that from   the definition of $\tilde W(F)$, if  $s\in {\rm supp}(1-\Phi)$, $ \tilde W(F)-E_{-}> \delta E + O(F)$.
Accordingly   for $F$ and $\beta$ small,
\begin{equation}
\Im \mu_\theta^2 \left( \tilde W_\theta(F)-E_{-} \right) \leq - \beta \frac { \delta E}{2}
\end{equation}
 Thus we get
  \begin{align}\nonumber
\Im \mu_\theta\left( \tilde H_{0,\theta}(F)-z\right)& \mu_\theta = \Im \mu_\theta(T_{s,\theta}\mu_\theta) +T_2(\theta)  + \Im \mu_\theta^2 T_u  + \Im \mu_\theta^2\left(\tilde  W_\theta(F)-E_{-}\right) + \\&  \nonumber 
    \Im \mu_\theta^2( E_- -z)
 \leq \Im \mu_\theta(T_{s,\theta}\mu_\theta)  - \beta \frac { \delta E}{2} + \Im \mu_\theta^2( E_-  + \lambda_0 -z)
\end{align}
and since
\begin{equation}
\label{ts 1}
\Im\mu_\theta(T_{s,\theta}\mu_\theta)=  O(\beta^{2} F),
\end{equation}
 then for  $z\in \nu_\theta$, for $F $  and $\beta$ small enough 
$$\Im \mu_\theta\left( \tilde H_{0,\theta}(F)-z\right)\mu_\theta \leq -\beta \frac{\delta E}{4}  + \Im \mu_{\theta}^{2}(E_{-}+ \lambda_{0}-z)<0.$$
Thus the proof of the proposition follows  (see \cite{PB} or \cite{G} for more details). 
\end{proof}

By standard arguments the Proposition \ref{th} holds for $ 0<\vert \theta \vert \leq \beta_0,0<  \Im \theta$ and $0<F<F_0$.
\section{Meromorphic continuation of the resolvent}

 Under conditions of the Proposition \ref{th}. Set $V_\theta=V_{0, \theta}+ W_{\theta}(F)-\tilde{W}_{\theta}(F)$. 
Introduce the following  operator, let  $\theta  \in \mathbb C$, $\vert \theta \vert <  \beta_0 $, $0< F\leq F_0$,  $z\in\nu_{\theta}$ and 
\begin{equation}
\label{kteta}
K_{\theta}(F,z)=V_\theta({\tilde H}_{0,\theta}(F)-z)^{-1}.
\end{equation}
Then 
\begin{proposition} \label{l2}
\begin{itemize} Suppose (h1-3). Then there exists $ 0< \beta_0\leq \min\{ \delta E,  a_0 \}$ and $F_0>0$  such that for $ 0 <  \vert \theta \vert  \leq \beta_0, \Im \theta = \beta>0$, 
   and   $0<F\leq F_0$,   

\item [(i) ] $z \in \nu_\theta \to K_{\theta}(F,z)$ is  an analytic compact  operator valued function.
\item [(ii) ] For $z\in  \nu_\theta$, $Imz> 0$ large enough, $\|K_{\theta}(F,z)\| < 1$.
\end{itemize}
\end{proposition}
\begin{proof} Let us first  show that  under these conditions, $K_{\theta}(F,z),  \theta = i \beta,  \Im z >0 $ are compact operators, this allows to prove  the Proposition \ref{l2} i).
We know from  \eqref{dWtheta} and (h1-2)  that  
$V_\theta  = O(\frac{F}{\vert s \vert^{\epsilon -1}}) $ as $ s \to \pm \infty$.  Denote by $\mathbb I_{\cal H}$ the identity operator on  the space ${\cal H}$. Let $h_0= -\partial_s^2  \otimes \mathbb I_{L^2(0,d)} +   \mathbb I_{L^2(\mathbb R)} \otimes T_u $. Then    the decay property of $V_\theta$ imply that  the operator      $ V_\theta (h_0 - z)^{-1}$ is compact    for $\Re z < \lambda_0$  (see e.g.   \cite{BG} or \cite{Herbst})

We have 
\begin{equation} \label{E2}
\tilde{H}_{0,\theta}(F)- h_0=    \partial_{s}(1-G_{\theta}) \partial_{s} + S_\theta + \tilde W_\theta(F)  
\end{equation}
where $ G_\theta$, $S_\theta$  are given  respectively by  \eqref{11},  \eqref{Seta} and 
$\tilde W_\theta(F) $ is defined  in the  Section 3.
Hence,
\begin{equation} \label{e1}
K_{\theta}(F,z)= 	V_\theta (h_{0}-z)^{-1}+ V_\theta(h_{0}-z)^{-1}\big(  \tilde W_\theta(F) + \partial_{s}(1-G_{\theta}) \partial_{s} + S_\theta \big)(\tilde H_{0, \theta}(F)-z )^{-1}.
 \end{equation}
 Let first show that the operator $  I_1:= V_\theta (h_{0}-z)^{-1} \tilde W_\theta(F) (\tilde H_{0, \theta}(F)-z )^{-1}$ is compact. This follows from  the Herbst's argument \cite{Herbst}. Indeed  let $ l(s):= (1+ \vert s\vert^{2})^{1/2} $, then
  \begin{multline} \label{e11}
  I_1=V_\theta l (h_{0}-z)^{-1} \frac{\tilde W_\theta(F)}{ l}(\tilde H_{0, \theta}(F)-z )^{-1}+ \\
 V_\theta (h_{0}-z)^{-1}[l,h_{0}](h_{0}-z)^{-1} \frac{\tilde  W_\theta(F)}{ l}(\tilde H_{0, \theta}(F)-z )^{-1}.
\end{multline}
Where $[A,B]$ denotes the commutator of the operators $A$ and $B$.
 The operator $V_\theta l (h_{0}-z)^{-1} $ is compact since  $V_\theta(s,u) l(s) = O(\frac{F}{ \vert s \vert ^{ \varepsilon- 3}})$ 
as  $ s  \to  \pm  \infty$ and $\varepsilon >3$.  This holds true for the first operator of the r.h.s. of  \eqref{e11}  since  
 $\frac{ \tilde W_\theta(F)}{ l}(\tilde H_{0, \theta}(F)-z )^{-1}$ is a bounded operator.   
 In the other hand by the closed graph theorem,  $  [l,h_{0}](h_{0}-z)^{-1}  $  is a bounded operator.
 So  it follows  that  the second operator of  of the r.h.s. of  \eqref{e11}  and then $I_1$  is also compact.
 
Set  $I_2= V_\theta(h_{0}-z)^{-1} \partial_{s}(1-G_{\theta}) \partial_{s} (\tilde H_{0, \theta}-z )^{-1} .$
 Let us show $ \partial_{s}(1-G_{\theta}) \partial_{s} (\tilde H_{0, \theta}(F)-z )^{-1}, \;  \Im z  >0  $ is a bounded
 operators.  In view of  the Corollary \ref{l1}, we are left to show that  $ \partial_{s}(1-G_{\theta}) \partial_{s} (\tilde H_{0}(F)-z )^{-1}$ is  bounded. We have 
 $$ \partial_{s}(1-G_{\theta}) \partial_{s} (\tilde H_{0}(F)-z )^{-1} =  \partial_{s}(1-G_{\theta}) \partial_{s} (H_{0}-z )^{-1} 
 +   \partial_{s}(1-G_{\theta})\partial_{s} (H_{0}-z )^{-1} \tilde W(F) (\tilde H_{0}(F)-z )^{-1}. $$ 
 By the closed graph theorem $\partial_{s}(1-G_{\theta}) \partial_{s} (H_{0}-z )^{-1}$  is bounded. Now  the second term of the r.h.s. of this equality can be written as 
  \begin{multline}
 \partial_{s}(1-G_{\theta}) \partial_{s}  l (H_{0}-z )^{-1}  \frac{\tilde W(F)}{l} ( \tilde H_{0}(F)-z )^{-1}  
+\\   \partial_{s}(1-G_{\theta}) \partial_{s}  (H_{0}-z )^{-1}[H_0, l] (H_{0}-z )^{-1} \frac{ \tilde W(F)}{l} (\tilde H_{0}(F)-z )^{-1}. 
\end{multline}
 Notice that under  (h2) then $(1-G_{\theta} )l$ as well as $(1-G_{\theta}) l'$ are bounded functions. Hence following the same arguments as above we are done. Evidently $I_3=  V_\theta(h_{0}-z)^{-1} S_{\theta}  (\tilde H_{0, \theta}-z )^{-1} $ is also  a compact operator. This proves our claim.

 In the other hand   the Proposition \ref{th} implies ii).
  \end{proof}
 We now prove the first part of the Theorem  \ref{t0}.\\
   Let $0<F\leq F_0$ and $E<0$. For  $ 0<\vert \theta \vert < \beta_0, \Im \theta >0$   by   the Proposition \ref{th}, the Lemma \ref{l2} and the Fredholm alternative  theorem,   the operator $ \mathbb I_{L^2(\Omega)}+  K_{\theta}(F,z)$  is invertible for all    $z \in \nu_\theta \setminus \mathcal R$ where $\mathcal R$ is a discrete  set. In 
 the bounded operator sense, we have  
\begin{equation} \label{Res}
( H_\theta(F)-z)^{-1}= ( \tilde H_{0,\theta}(F)-z)^{-1} \big(\mathbb I_{L^2(\Omega)}+  K_{\theta}(F,z) \big)^{-1}.
\end{equation}
This implies that $ \nu_\theta \setminus \mathcal R \subset \rho( H_\theta(F))$.

Choose $\beta_0$ so small that  there exists a   dense subset  of analytic vectors associated to
the transformation $U_\theta$ in $\vert \theta \vert <\beta_0 $ (see \cite [Remark (3.3)] {BG}). We denote  this set by $ \mathcal A$.
Then standards arguments of the distortion theory, and  \eqref{Res}
imply that for all $\varphi \in \mathcal A$
 \begin{equation} \label{Rzz}
{\cal R}_\varphi(z)=   \big(  ( H(F)-z)^{-1}\varphi, \varphi\big), \Im z >0
\end{equation}
has an  meromophic extension  in $ \nu_\theta$ 
given by 
$$ {\cal R}_\varphi(z)=  \big(  ( \tilde H_{0,\theta}(F)-z)^{-1} \big( \mathbb I_{L^2(\Omega)} +  K_{\theta}(F,z) \big)^{-1}\varphi_\theta,\varphi_{\bar \theta} \big)$$ 
 We define  the resonances of the operator $H_\theta(F)$ as the poles of ${\cal R}_\varphi$. They  are  locally $\theta $-independent and   in view of \eqref{Res} they  are the discrete eigenvalues of the operator $H_\theta(F)$.   
\qed


\section{Resonances}

We want to prove the following result, let $\theta= i \beta$,  $ 0< \beta  \leq \beta_0$, $\beta_0$ as in the Proposition \ref{l2} and the proof of the Theorem  \ref{t0}  i). Then 
\begin{proposition}
\label{pT}
 Let $E_{0}$ is a negative eigenvalue of $H$ with multiplicity $n$. There exists $ F_0>0$ such that for $0<F\leq F_0$,$H_\theta(F)$  has  exactly $n$ eigenvalues   denoted by  $Z_0, ..., Z_{n-1}$,  satisfying $\dlim_{F\rightarrow 0}|Z_j-E_0|=0.$
\end{proposition}

 The Proposition \ref  {pT} implying the Theorem \ref{t0} ii).
\begin{proof}
Following \cite[Section 5] {BG}, to  prove the Proposition \ref{pT}  we  only have to  prove   that under 
(h1-3), for   $E<0$, $ \theta =i \beta, 0<\beta \leq \beta_0$.  Then
\begin{equation} \label {14}
  \Vert K_{\theta}(F,z)- K(z) \Vert \rightarrow 0  \;  { \rm{as}} \; F\rightarrow 0 
\end{equation}
uniformly in $z\in  {\mathcal K}$,  where  $ {\mathcal K}$ is  a compact subset of $\nu_\theta\bigcap \rho(H_0)$.
By a continuity argument it is sufficient to  prove that  for some $z_0 \in {\mathcal K}$.  We have 
\begin{align}\notag
K_{\theta}(F,z_0)- K(z_0) =& (V_{0,\theta}-V_0) (H_{0}-z_0)^{-1}  + (W_{\theta}(F)-\tilde W_{\theta}(F)) (\tilde{H}_{0,\theta}(F)-z_0)^{-1}  + \\ \label {15}
&V_{0,\theta}((\tilde{H}_{0,\theta}(F)-z_0)^{-1}-(H_{0}-z_0)^{-1}).
\end{align}
From \eqref{15}, the proof needs several steps. First we consider the two first term of the r.h.s. of \eqref{15}. 
In view of     \eqref{15} and  \eqref{unitaire}, we can see that   $$ \Vert V_{0,\theta}-V_0 \Vert_\infty \leq C \big(   \Vert \gamma_{\theta}-\gamma \Vert_\infty
 +   \Vert \gamma'_{\theta}-\gamma' \Vert_\infty+  \Vert \gamma^{''}_{\theta}-\gamma^{''} \Vert_\infty
\big)$$
 for some constant $C>0$. Hence we can estimate each term of the r.h.s. of this last inequality  by using   Taylor expansions w.r.t. $\theta$ (see e.g. \eqref{gamtheta}.   Then we find
 
 $$\Vert\gamma_{\theta}-\gamma \Vert_\infty,\Vert \gamma'_{\theta}-\gamma' \Vert_\infty,  \Vert \gamma^{''}_{\theta}-\gamma^{''} \Vert_\infty = O(\beta F ^{\varepsilon-1})$$
 and  then  $\Vert  (V_{0,\theta}-V_0) (H_{0}-z_0)^{-1}  \Vert \to 0$ as $F\to 0$. The Proposition  \ref{th}   and \eqref{dWtheta} imply that $\Vert (W_{\theta}(F)-\tilde W_{\theta}(F)) (\tilde{H}_{0,\theta}(F)-z_0)^{-1}\Vert   \to 0$ as $F\to 0$.
Hence we are left to show that this is true for the third term of the r.h.s of \eqref{15} i.e.
$$   \delta H (z_0):= V_{0,\theta}((\tilde{H}_{0,\theta}(F)-z_0)^{-1}-(H_{0}-z_0)^{-1}). $$
 Let $H_{0,\theta}$ be the operator introduced in the Section 3. We use the estimate
\begin{align} \notag
\Vert  \delta H (z_0)& \Vert \leq  \Vert V_{0,\theta} ((H_{0,\theta}-z_0)^{-1}- (\tilde{H}_{0,\theta}(F)-z_0)^{-1})  \Vert + \\
&\Vert  V_{0,\theta}((H_{0,\theta}-z_0)^{-1}-{H}_{0}-z_0)^{-1})  \Vert.
\end{align} 
 First we consider 
$$  \delta_1 H (z_0)= V_{0,\theta}(H_{0,\theta}-z_0)^{-1} \tilde W_{\theta}(F) (\tilde{H}_{0,\theta}(F)-z_0)^{-1}.$$
Following arguments of the proof of the Proposition \eqref{l2} we have 
\begin{align} \notag
\delta_1 H (z_0) = &V_{0,\theta} l (H_{0,\theta}-z_0)^{-1} \frac { \tilde W_{\theta}(F)}{l} (\tilde{H}_{0,\theta}(F)-z_0)^{-1} + \\
&V_{0,\theta} [(H_{0,\theta}-z_0)^{-1},l] \frac { \tilde W_{\theta}(F)}{l} (\tilde{H}_{0,\theta}(F)-z_0)^{-1}.
\end{align} 
  Clearly  $ \Vert \frac { \tilde W_{\theta}(F)}{l} \Vert  \to 0 $ as $ F\to 0$.   
   In the other hand  we know that   $V_{0,\theta}, V_{0,\theta} l$  are  bounded  uniformly w.r.t.
   $F$. In view of the Propositions  \ref{HOAp} and   \ref{th}, this also holds for the resolvents
   $(H_{0,\theta}-z_0)^{-1}$ and $(\tilde{H}_{0,\theta}(F)-z_0)^{-1}$. Let us show  that this  is again true  for $ [(H_{0,\theta}-z_0)^{-1},l]$ then this will imply that   $ \Vert \delta_1 H (z_0)\Vert \to 0 $ as $F\to 0$.\\
 We have, 
  $$ [(H_{0,\theta}-z_0)^{-1},l]= -(H_{0,\theta}-z_0)^{-1}( (G_\theta l')' +  2l'G_\theta  \partial_s)    (H_{0,\theta}-z_0)^{-1}.$$
 and   $\Re G_\theta \geq 0, \Im G_\theta \leq 0 $.  It is  then sufficient   to show that $\Vert  (\Re G_\theta)^{1/2}   \partial_s    (H_{0,\theta}-z_0)^{-1}\Vert$ and 
 $\Vert(-\Im G_\theta )^{1/2}  \partial_s    (H_{0,\theta}-z_0)^{-1}\Vert$   are  uniformly 
bounded w.r.t. $F$ for $F$ small.  Recall that  $ H_{0, \theta} = -\partial_s G_\theta \partial_s   +T_u +S_\theta$.  Then  our last claim follows from the arguments evoked above   and   that   for $\varphi \in L^2(\Omega)$, $\Vert  \varphi \Vert =1$,
\begin{align} \notag
\Vert  (\Re G_\theta)^{1/2}  \partial_s    (H_{0,\theta}-z_0)^{-1} \varphi \Vert \leq 
 \Re \big (   (H_{0,\theta}-z_0)^{-1} \varphi , \varphi \big) ) \\  \notag -\Re \big (   (H_{0,\theta}-z_0)^{-1} \varphi, (S_\theta + z_0 )   (H_{0,\theta}-z_0)^{-1}\varphi \big) 
\end{align}
 and 
\begin{align}\notag
\Vert  (-\Im G_\theta)^{1/2}  \partial_s    (H_{0,\theta}-z_0)^{-1} \varphi \Vert =
 -\Im \big (   (H_{0,\theta}-z_0)^{-1} \varphi , \varphi \big) ) \\\notag  +  \Im \big (   (H_{0,\theta}-z_0)^{-1} \varphi, (S_\theta + z_0 )   (H_{0,\theta}-z_0)^{-1}\varphi \big). 
\end{align}
Evidently  this imply that there exists a constant $c'>0$ s.t. 
$$\Vert  (\Re G_\theta)^{1/2}  \partial_s    (H_{0,\theta}-z_0)^{-1} \Vert, \Vert  (-\Im G_\theta)^{1/2}  \partial_s    (H_{0,\theta}-z_0)^{-1} \Vert \leq \Vert   (H_{0,\theta}-z_0)^{-1}\Vert  + c'\Vert   (H_{0,\theta}-z_0)^{-1}\Vert ^2.$$
Now consider
$$  \delta_2 H (z_0)= V_{0,\theta}(H_{0,\theta}-z_0)^{-1}( T_s-T_{s,\theta}) ({H}_{0}-z_0)^{-1}.$$
For $F$ and $\beta$ small, $ T_{s}- T_{s,\theta}=  \partial_s G \partial_s + S_\theta$ with
 $G=  (\gamma_\theta - \gamma)G_1 + i\beta f' G_2 $  where  $ G_1, G_2$ are  uniformly  bounded functions w.r.t. $F$ (see e.g. \eqref{g} and \eqref{ts teta}.\\
  Moreover we know from \eqref{gamtheta} that $\gamma_\theta - \gamma= O(\beta F^{\varepsilon-1})$,  and 
  $ S_\theta= O(\beta F^2)$ then
   \begin{align} \label {esta}
  \Vert V_{0,\theta}(H_{0,\theta}-z_0)^{-1} (\partial_s (\gamma_\theta - \gamma)G_1  \partial_s +& S_\theta )({H}_{0}-z_0)^{-1} \Vert  \leq  \\ \notag
  C  \beta \big( F^{\varepsilon} \Vert (H_{0,\theta}-z_0)^{-1} \Vert \Vert \partial_s G_3  \partial_s (H_{0}-z_0)^{-1}  \Vert &+  \\ &\notag F^{2} \Vert (H_{0,\theta}-z_0)^{-1}\Vert \Vert  (H_{0}-z_0)^{-1} \Vert \big)
\end{align}
  for some  constant $C>0$. Where $G_3$   is uniformly bounded w.r.t. $F$.  So this term   vanishes as $ F\rightarrow 0 $. 

To study the  second term, we use a  different strategy.   We have  
\begin{align} \label {AA}
  V_{0,\theta}(H_{0,\theta}-z_0)^{-1}& \partial_s f'G_2  \partial_s  ({H}_{0}-z_0)^{-1}= \\ \notag
  & V_{0,\theta}(H_{0}-z_0)^{-1} (H_{0}-z_0) (H_{0,\theta}-z_0)^{-1} \partial_s  f'G_2  \partial_s({H}_{0}-z_0)^{-1}
  \end{align}
We know that the operator $ V_{0,\theta}(H_{0}-z)^{-1} $ is a  compact  operator (see e.g. the proof of the Lemma \ref{l2}) then to prove that  the operator in the l.h.s. of \eqref  {AA}  converges in the norm sense 
to $0_{\mathcal B(L^2(\Omega))}$ as $F\to 0$,   it is sufficient to show that $ (H_{0}-z_0) (H_{0,\theta}-z_0)^{-1} \partial_s  f'G_2  \partial_s({H}_{0}-z_0)^{-1}$ converges strongly to $0_{\mathcal B(L^2(\Omega))}$ as $F\to 0$. \\

 Recall that  $\mathcal{C}=\{\varphi = \tilde \varphi \lfloor_ \Omega, \tilde \varphi \in C_{0}^{\infty}(\mathbb R^2); \tilde \varphi \lfloor_ {\partial \Omega} =0 \}$ is a core of $H_0$, thus for $z\in \rho(H_0)$, $\mathcal C'=(H_0-z_0)\mathcal C$ is dense in $L^{2}(\Omega)$.  Set  $\psi=(H_0-z_0)\varphi$, $ \varphi \in \mathcal C$. \\

  
Since the field $f$  is choosed s.t. $f'$ has  support contained  in $ \vert s \vert > c'/F$ for some  $c'>0$  then  $\dlim_{F\rightarrow 0}\|\partial_s  f'G_2  \partial_s \varphi\|=0$. By using standard arguments of the perturbation theory and the proposition \ref{HOAp},  for $F$ small, the operator $(H_{0}-z_0) (H_{0,\theta}-z_0)^{-1}$  has a norm which  is uniformly bounded w.r.t $F$. This proves our claim on $\mathcal C'$.

 In the other hand $ \partial_s  f'G_2  \partial_s  ({H}_{0}-z_0)^{-1}$ is bounded operator with a norm  uniformly bounded w.r.t. $F$, for $F$ small.   Then the strong convergence follows.
\end{proof}

\section{ Concluding remarks}
In this last section we would like  to give some remarks about the field regime related to this problem.
  Let us mention  that the first result was given by P. Exner in   \cite{Ex}, for $  \eta =  \frac {\pi}{2}$  and $ \alpha_0 = 0$. 
  But P. Exner did not consider the  question of  existence of resonances. \\
  This issue was addressed  by us in  \cite{BG}. 
In  this paper  we  have considered  the conditions   $ \vert \eta \vert < \frac {\pi}{2} $   and $ \vert  \eta  - \alpha_0 \vert<\frac {\pi}{2}$. Roughly speaking this corresponds to the classical picture of the Stark effect for  one dimensional Schr\" odinger operators with local potential i.e. the field interaction $W(F) \to -\infty $ as $s\to  -\infty$ and  $W(F) \to \infty $ as $s\to  \infty$. For  any  negative reference energy the   non trapping region coinciding  with  a neighbourhood of 
$s=-\infty$.  
In this case we prove an analog of the Theorem \ref {t0}. \\
Evidently   the regime  $ \vert \eta \vert > \frac {\pi}{2} $   and $ \vert  \eta  - \alpha_0 \vert>\frac {\pi}{2}$  is  a symmetric  to the above mentioned case. 

Suppose   now that $ \vert   \eta \vert  >   \frac {\pi}{2}$  and    $\vert  \eta  - \alpha_0 \vert< \frac {\pi}{2}$. Then 
 from \eqref{W201}   and \eqref{W202},   clearly $W(F) \to \infty $ as $ s \to \pm \infty$ i.e.  it     is    a confining potential.  By using standard arguments (see e.g. \cite{reed simonIV}) it easy to see that $H(F)$ has a compact resolvent and then only discrete spectrum. \\
 Indeed consider  the following operator in $L^2(\Omega)$. Let $F>0$,
 $$ h(F)= T_s+ T_u + w(F).$$
 where $ w(F) = F \cos(\eta) s$   if  $s\leq 0$ and $ w(F) = F \cos(\eta - \alpha_0) s$ if  $s> 0$. Then 
by (h1) and (h2), there exists a  strictly positive constant $c$ such that  in the form sense we have 
\begin{equation} \label{f1}
h(F) \geq h_1(F) := (-c \partial_s^2  +  w(F)) \otimes  \mathbb I_u +   \mathbb I_s \otimes  -\partial_u^2.   
\end{equation}
 But  the operator  $-\partial_s^2  +  w(F) $ has a compact resolvent  \cite{reed simonIV}. Let   $\{ p_n,  n\geq 0\}$ the  eigen-projectors corresponding to the  transverse  modes $\{ \lambda_n,  n\geq 0\}$. Since  
 
 $$  (h_1(F)  +1)^{-1}= \oplus_{n\geq 1} (-c\partial_s^2  +  w(F)  + \lambda_n +1)^{-1}) \otimes p_n$$
 Then $  (h_1(F)  +1)^{-1}$ is a compact operator as a norm limit of compact operators. Hence $h_1(F)$ satisfies
 the Rellich criterion   \cite{reed simonIV} and  in view  of \eqref{f1} it is also true for $h(F)$ so 
 $  (h(F)  +1)^{-1}$ is compact. Since $H(F) -h(F)$ is bounded then $  (H(F)  +1)^{-1}$ is also  a compact operator.

\end{document}